\documentclass[12pt]{amsart}
\usepackage{amsfonts}

\usepackage{amsxtra}
\usepackage{latexsym}
\usepackage{amssymb}
\usepackage{amsthm}
\usepackage{newlfont}
\usepackage{verbatim}
\usepackage{amsmath}
\usepackage[usenames,dvipsnames]{xcolor}

\newtheorem{thm}{Theorem}[section]
\newtheorem{lem}[thm]{Lemma}
\newtheorem{prop}[thm]{Proposition}
\newtheorem{rmk}{Remark}[section]
\newtheorem{cor}[thm]{Corollary}

\newcommand{\qq}{\quad\quad}

\newcommand{\nf}{\infty}

\newcommand{\al}{\alpha}

\newcommand{\ga}{\gamma}
\newcommand{\Ga}{\Gamma}
\newcommand{\de}{\delta}

\newcommand{\ep}{\epsilon}

\newcommand{\tht}{\theta}

\newcommand{\la}{\lambda}

\newcommand{\si}{\sigma}

\newcommand{\vp}{\varphi}
\newcommand{\om}{\omega}

\newcommand{\rn}{\mathbb R^n}

\newcommand{\wh}{\widehat}

\newcommand{\p}{\partial}

\newcommand{\f}{\frac}

\newcommand{\tf}{\tfrac}

\newcounter{question}
\newcommand{\qt}{%
        \stepcounter{question}%
        \thequestion}
\newcommand{\bq}{\fbox{Q\qt}\ }

\newcommand{\wtd}{\widetilde}
\newcommand{\bbr}{\mathbb R}

\begin{document}

\title[Bilinear Maximal Bochner-Riesz]
{On Bilinear Maximal Bochner-Riesz Operators}

\author{Danqing He}
\address{Department of Mathematics, 
Sun Yat-sen (Zhongshan) University, 
Guangzhou, 510275, 
P.R. China}

\email{hedanqing35@gmail.com}


\thanks{{\it Keywords and phases:} Bilinear multipliers, bilinear maximal Bochner-Riesz operators,
Fourier series, wavelets, multilinear operators}


\begin{abstract}
We prove that the bilinear maximal Bochner-Riesz operator
$T_*^\la$ is bounded from $L^{p_1}(\rn)\times L^{p_2}(\rn)$ to 
$L^p(\rn)$ for appropriate $(p_1,p_2,p)$ when $\la>(4n+3)/5$. 
\end{abstract}

\maketitle



 
 \section{Introduction}

 
 Let $\mathbb T^n$ be the torus $[0,1]^n$.
 For a function $f$ in $L^1(\mathbb T^n)$, its Fourier coefficient $a_k$ for $k\in\mathbb Z^n$ 
 is defined by
\begin{equation}\label{FS1}
 a_k=\int_{\mathbb T^n}f(x)e^{-2\pi ik\cdot x}dx.
\end{equation}
 The series $\sum_{k\in\mathbb Z^n}a_ke^{2\pi ik\cdot x}$ is called the Fourier series related to $f$.
 
The pointwise convergence of the Fourier series \eqref{FS1} has been a central problem of Fourier analysis. Carleson's 
celebrated work \cite{Carleson1966},   answering  Lusin's conjecture,   shows 
that the pointwise convergence is valid a.e. for   $f$ in $L^2(\mathbb T)$. This 
was generalized later by Hunt \cite{Hunt1968} for any $f$ in $L^p(\mathbb T)$ with $1<p\le \nf$.
These pointwise convergence results were   reproved by Fefferman \cite{Fefferman1973}, and Lacey and Thiele \cite{Lacey2000}.

For higher dimensions the Fourier series could naturally be interpreted as the limit
of the spherical partial sum  $\sum_{|k|\le N}a_ke^{2\pi ikx}$. Unfortunately Fefferman's counterexample \cite{Fefferman1971} indicates
 that the pointwise convergence fails for such sums  when $n\ge2$ and $p\neq 2$. 
 For this reason, it is   natural to consider  smoother versions of spherical sums,  known as
the Bochner-Riesz means, for which several analogous pointwise convergence results have been obtained; see, for instance Carbery \cite{Carbery1983}, Christ \cite{Christ1985}, Carbery, Rubio~de Francia  and   Vega~\cite{Carbery1988}, and Tao
\cite{Tao2002}.

We can also consider analogous bilinear questions. In  
the bilinear theory, developed in the past decades, we  study the restriction of the  output of 
linear   operators on the diagonal, when the input is of tensor product form.
 In our case, we ask what type of convergence can we obtain for the operator
\begin{equation}\label{av}
A^\la_{t}(f,g)(x)
=\int_{\mathbb R^n}\int_{\mathbb R^n}\wh f(\xi)\wh g(\eta)\big(1-(|t\xi |^2+|t \eta |^2)\big)_+^{\la}e^{2\pi i x\cdot (\xi+\eta)}d\xi d\eta,
\end{equation}
which can be regarded as $B_{1/t}^\la(f\otimes g)(x,x)$ with $B^\la_{1/t}$ the Bochner-Riesz operator
on $\bbr^{2n}$ and $x\in\rn$. For test functions  we should have
$A^\la_t(f,g)\to fg$ as $t\to0$ in   
the $L^p$ or pointwise sense.  Grafakos and Li 
\cite{Grafakos2006a} and Bernicot Grafakos, Song and Yan~\cite{Bernicot2015a}
have proved some partial positive results
 for $\la=0$ and $\la>0$ respectively
concerning the $L^p$ convergence. 

In this paper, we are concerned with the pointwise 
convergence of the means \eqref{av}, in particular with the 
boundedness of the bilinear maximal Bochner-Riesz 
operator, which of course implies the boundedness of the
bilinear Bochner-Riesz operators. 
The bilinear 
maximal Bochner-Riesz 
 operator  for $\la>0$
 is defined as
\begin{equation}\label{bBR}
T^{\la}_*(f,g)(x)=\sup_{t>0}
\Big|\int_{\mathbb R^n}\int_{\mathbb R^n}m(t\xi,t\eta)\wh f(\xi)\wh g(\eta)e^{2\pi i x\cdot (\xi+\eta)}d\xi d\eta\Big|,
\end{equation}
where $m({\xi,\eta})=m^{\la}(\xi,\eta)=(1-(|\xi |^2+|\eta |^2))_+^{\la}$,
which is equal to $(1-(|\xi |^2+|\eta |^2))^{\la}$ when $|(\xi,\eta)|\le 1$ and $0$ when
$|(\xi,\eta)|> 1$.
For simplicity we   occasionally  denote  $T^{\la}_*$ by   $T_*$ 
when there is no confusion. $T_*$ is a natural 
generalization of the (linear) maximal Bochner-Riesz  operator from which it naturally 
inherits its name.

Our main theorem is as follows.
\begin{thm}\label{06151}
When $\la>(4n+3)/5$, $T_*^\la$
is bounded from $L^{p_1}\times L^{p_2}$
to $L^p$ when
$p>(2n-11)/(10\la-6n-17)$,
$p_1,\ p_2>(4n-22)/(10\la-6n-17)$
and $1/p=1/p_1+1/p_2$.

\end{thm}

To prove this theorem, we need to interpolate between 
the positive result for $\la$ above the critical index (Proposition \ref{Uni})
and the $L^2\times L^2\to L^1$ boundedness (Theorem \ref{Main}) for $\la$ down to
$(4n+3)/5$. The former  of these results is standard but the latter is novel and constitutes the 
main contribution of this paper. 
Assuming these two results, we prove Theorem \ref{06151}
via the standard interpolation technique known from the  linear case; 
we     outline some ideas of this interpolation technique below and omit the  details,
which can be found in    \cite{Stein1956}
and \cite{Stein1971a}.

\begin{proof}[Proof of Theorem \ref{06151}]
First we   (bi)linearize this maximal operator $T_*^\la$. Let $\mathcal A$
be the class of nonnegative measurable functions on $\rn$ with finitely many distinct
values. For each function $R\in\mathcal A$, we can define the bilinear operator
$T_{R(x)}(f,g)=T^\la_{R(x)}(f,g)$
by the integral
$$
\int_{\rn}\int_{\rn}(1-R(x)^2(|\xi|^2+|\eta|^2))^{\la}_+\wh f(\xi)
\wh g(\eta)e^{2\pi ix\cdot \xi}d\xi d\eta.
$$
It is not hard to verify that for Schwartz functions $f$ and $g$ we have
$$
\sup_{R(x)\in\mathcal A}\|T_R(f,g)\|_p
=\|T_*(f,g)\|_p.
$$
Hence we know that $T_R$ is bounded for all indices we
have proved for $T_*$. And our claim will be established if we can prove the corresponding results 
for all $T_R$ with $R\in\mathcal A$ such that the constants involved are 
independent of $R$.

The advantage of $T_R^\la$ is that it is bilinear so that we are allowed to apply interpolation
results. The specific one we use here is Theorem 7.2.9 of \cite{Grafakos2014a}.
We can verify that $T^\la_R(f,g)$, as can be defined for all complex numbers $\la$, is 
analytic in $\la$ and actually admissible when $f$ and $g$ are simple functions.
We know that $T_R^{\la_0}$ is bounded from $L^2\times L^2$
to $L^1$ for $\la_0$ whose real part
is strictly great than $(4n+3)/5$, and 
$T_R^{\la_1}$ is bounded from $L^{q_1}\times L^{q_2}$
to $L^q$ with Re$\la_1>n-1/2$, $q_1,q_2>1$ and $1/q=1/q_1+1/q_2$.
The bounds for these two cases are of admissible growth  in the imaginary parts of
$\la_j$ for $j=0,1$, and they both are independent of $R\in\mathcal A$. Consequently 
we can interpolate between these two points using Theorem 7.2.9 of \cite{Grafakos2014a}
and obtain
that $\|T_R^\la(f,g)\|_{L^{p}}\le C\|f\|_{L^{p_1}}\|g\|_{L^{p_2}}$
for $p_1,\ p_2>(4n-22)/(10\la-6n-17)$
and $1/p=1/p_1+1/p_2$,
which   implies the claimed conclusion.
\end{proof}

 \begin{rmk}
We have all strong boundedness points for $T_*$ when $\la$ is larger than the critical index
$n-1/2$ by Proposition \ref{Uni}, which is greater than $(4n+3)/5$  when $n\ge 6$.
 \end{rmk}

 As a  corollary of Theorem \ref{06151}, we obtain the pointwise
 convergence, as $t\to0$, of the operator $A^\la_{t}(f,g)(x)$, which we denote
 by $A_{t}(f,g)(x)$ as well.
 
\begin{prop}
Suppose $\la>\min\{(4n+3)/5,\ n-1/2\}$, then
for $f\in L^{p_1}$ and $g\in L^{p_2}$ with $p_1,p_2$ as in Theorem 
\ref{06151}, 
we have
\begin{equation}\label{pc}
\lim_{t\to0}A_t(f,g)(x)\to f(x)g(x) \qq\text{ a.e..}
\end{equation}
\end{prop}
The proof of this proposition is similar to the linear case, but we sketch it here for completeness.
\begin{proof}
It is easy to establish \eqref{pc} when both $f$ and $g$ are Schwartz functions.
To prove \eqref{pc} for $f\in L^{p_1}$ and $g\in L^{p_2}$ it suffices to 
show that for any given $\de>0$ the set
$E_{f,g}(\de)=\{y\in\mathbb R^n:O_{f,g}(y)>\de\}$ has measure $0$,
where
$$
O_{f,g}(y)=\limsup_{\tht\to0}\limsup_{\ep\to0}\big|A_{\tht}(f,g)(y)-A_\ep(f,g)(y)\big|.
$$
For any positive number $\eta$ smaller than $\|f\|_{L^{p_1}},\, \|g\|_{L^{p_2}}$, there exist  Schwartz functions $f_1=f-a$ and $g_1=g-b$
such that both $\|a\|_{L^{p_1}}$, and $\|b\|_{L^{p_2}}$ are bounded by $\eta$.
We observe that
$$|E_{f,g}(\de)|\le |E_{f_1,g_1}(\de/4)|+|E_{a,g_1}(\de/4)|+|E_{f_1,b}(\de/4)|+|E_{a,b}(\de/4)|.$$
Notice that
$|E_{f_1,g_1}(\de/4)|=0$ since \eqref{pc} is valid for $f_1,g_1$. 
To control the remaining three terms, we observe that, for instance, 
\begin{align*}
|E_{a,g_1}(\de/4)|\le&|\{y:\ 2T_*(a,g_1)(y)>\de/4\}|\\
\le& C\Big(\f{\|a\|_{L^{p_{1}}}\|g_1\|_{L^{p_2}}}\de\Big)^p\\
\le &C\Big(\f{\eta\|g\|_{L^{p_2}}}\de\Big)^p,
\end{align*}
where the last term goes to $0$ as $\eta\to0$ since $g$ and $\de$ are fixed.
\end{proof}

 The boundedness of $T^\la_*$ when $\la>n-\tf12$ is straightforward from its kernel, 
 whose proof we postpone
 to  next section, so what remains is  
that $T^\la_*$ is bounded from $L^{2}\times L^{2}$ to $L^1$ for the claimed range of $\la$,
which will be the main focus of the rest of this paper.
 
 \begin{thm}\label{Main}
 When $\la>(4n+3)/5$, for $T_*$ in \eqref{bBR}
 we have that
 $$
 \|T_*(f,g)\|_{L^1}\le C\|f\|_{L^{2}}\|g\|_{L^{2}}.
 $$
 \end{thm}

Throughout this paper, we use the notation
$$\|T\|_{L^{p_1}\times L^{p_1}\to L^p}
=\sup_{\|f\|_{L^{p_1}}\le 1}\sup_{\|g\|_{L^{p_2}}\le 1}
\|T(f,g)\|_{L^p}.
$$
By a bilinear operator  $T$ related to the multiplier $\si$ we mean $T$ is defined by
$$T(f,g)(x)=\int_{\rn}\int_{\rn}\si(\xi,\eta)\wh f(\xi)\wh g(\eta)e^{2\pi ix\cdot (\xi+\eta)}d\xi d\eta$$
for all Schwartz functions $f$ and $g$.
 
\section{The Decompsitions}


Let us fix a nonnegative nonincreasing smooth function $\vp(s)$ on $\mathbb R$ such that $\vp(s)=1$ for $s\le 1/2$ and $\vp(s)=0$ for $s\ge1$. Define 
$\vp_j(s)=\vp(2^{j+1}(s+2^{-j}-1))$ for $j\ge 1$. 
Denote by $\psi_j(s)$ the function $\vp_1(s)$ when $j=0$ and
$\vp_{j+1}(s)-\vp_j(s)$ when $j\ge 1$. Notice that $\psi_0$ is supported in $(-\nf,3/4]$
and $\psi_j$ is supported in $[1-2^{-j},1-2^{-j-2}]$ for $j\ge1$.
Moreover $\sum_{j=0}^{\nf}\psi_j(s)=\chi_{(-\nf,1)}(s)$.

 We decompose the multiplier $m$  smoothly so that $m=\sum_{j\ge0}m_j$,
  where $m_j(\xi,\eta)=m(\xi,\eta)\psi_j(|(\xi,\eta)|) $ is supported in an annulus of the form
 $$
 \{(\xi,\eta)\in\mathbb R^{2n}: {1-2^{-j}} \le |(\xi,\eta)| \le   {1-2^{-j-2}} \}
 $$
 for $j\ge 1$ and $m_0$ is supported in a ball of radius $3/4$ centered at the
 origin. 
 
 Let 
$T_j(f,g)(x)=\sup_{t>0}
|\int_{\mathbb R^n}\int_{\mathbb R^n}\wh f(\xi)\wh g(\eta)m_j(t\xi,t\eta)e^{2\pi i x\cdot (\xi+\eta)}d\xi d\eta|,
$
then $T_*(f,g)(x)\le \sum_{j=0}^\nf T_j(f,g)(x)$.
We first present some trivial results on $T_*$ and $T_j$.
\begin{prop}\label{Uni} 
Assume $1<p_1,p_2\le\nf$, and 
$1/p=1/p_1+1/p_2$. Then for $\la>n-1/2$,
there exists a constant $C=C(p_1,p_2)$
such that
$\|T_*\|_{L^{p_1}\times L^{p_2}\to L^p}\le C$.
For any fixed $j$, there exists a constant $C_j$
such that
$\|T_j\|_{L^{p_1}\times L^{p_2}\to L^p}\le C_j$.

\end{prop}
\begin{proof}
Let us consider the kernel 
$K(y,z)=m^{\vee}(y,z)$ of $A_1$ defined in \eqref{av},  which satisfies that
$|K(y,z)|\le C(1+|y|+|z|)^{-(n+\la+1/2)}$ (see, for example, \cite{Grafakos2014b}),
hence for $\la>n-1/2$, we have
\begin{align*}
 |A_t(f,g)(x)| 
=&\bigg|\int_{\mathbb R^{2n}}t^{-2n}K(\f{x-y}t,\f{x-z}t)f(y)g(z)dydz\bigg|\\
\le &C(\vp_t*f)(x)(\vp_t*g)(x)\\
\le &CM(f)(x)M(g)(x),
\end{align*}
where $M$ is the Hardy-Littlewood maximal function, and $\vp_t(y)=t^{-n}\vp(y/t)$
with $\vp(y)=(1+|y|)^{-(n+\la+1/2)/2}$,
which is integrable when $\la >n-1/2$.
Then $T_*(f,g)(x)\le CM(f)(x)M(g)(x)$,
which implies that $\|T_*(f,g)\|_{L^p}\le C\|f\|_{L^{p_1}}\|g\|_{L^{p_2}}$ for
$1<p_1,p_2\le\nf$ with
$1/p=1/p_1+1/p_2$
in view of the boundedness of the Hardy-Littlewood maximal function.

We observe that each $m_j$ is smooth and compactly supported, hence for each fixed
$j$ a similar argument gives that $\|T_j\|_{L^{p_1}\times L^{p_2}\to L^p}\le C_j$.
 \end{proof}

With the aid of the preceding decomposition and the boundedness of
$T_j$, the study of the boundedness of
$T_*$ is reduced to locating the decay of $C_j$ in $j$; for the case $(p_1,p_2)=(2,2)$, this 
is contained in the following proposition.
 \begin{prop}\label{De}
 $T_j$
 satisfies that
 $$\|T_j\|_{L^2\times L^2\to L^1}\le C_nj2^{-j(\la-\f{4n+3}5)}.$$
 
 \end{prop}
 
 Let us prove Theorem \ref{Main} using Proposition \ref{De}.
 \begin{proof}[Proof of Theorem \ref{Main}]
 Since $T_*(f,g)(x)\le \sum_{j=0}^\nf T_j(f,g)(x)$, and the bound of $T_j$ has an exponential decay
 in $j$ when $\la>(4n+3)/5$ by 
 Proposition \ref{De},
 $\|T_*\|_{L^2\times L^2\to L^1}$ is finite.
 \end{proof}

 It suffices to consider the cases when $j$ is large.
 We  will use the wavelet decomposition of the multipliers as in \cite{Grafakos2015}. So we need 
 to introduce this decomposition due to \cite{Daubechies1988}, and the exact form we use here can be found
in \cite{Triebel2006}.


\begin{lem}[\cite{Triebel2006}]\label{wave}
For any fixed $k\in \mathbb N$ there exist real compactly supported functions $\psi_F,\psi_M\in \mathcal C^k(\mathbb R)$, 
which
satisfy that $\|\psi_F\|_{L^2(\mathbb R)}=\|\psi_M\|_{L^2(\mathbb R)}=1$
and $\int_{\mathbb R}x^{\al}\psi_M(x)dx=0$ for $0\le\al\le k$,
such that,  if $\Psi^G$ is defined by 
$$
\Psi^{G}(\vec x\,)=\psi_{G_1}(x_1)\cdots \psi_{G_{2n}}(x_{2n}) 
$$
for   $G=(G_1,\dots, G_{2n})$ in the set     
$$
 \mathcal I :=\Big\{ (G_1,\dots, G_{2n}):\,\, G_i \in \{F,M\}\Big\}  \, , 
$$
then the  family of 
functions 
$$
\bigcup_{\vec \mu \in  \mathbb Z^{2n}}\bigg[  \Big\{   \Psi^{(F,\dots, F)} (\vec  x-\vec \mu  )  \Big\} \cup \bigcup_{\la=0}^\nf
\Big\{  2^{\la n}\Psi^{G} (2^{\la}\vec x-\vec \mu):\,\, G\in \mathcal I\setminus \{(F,\dots , F)\}  \Big\}  
  \bigg]
$$
forms an orthonormal basis of $L^2(\mathbb R^{2n})$, where $\vec x= (x_1, \dots , x_{2n})$.  
\end{lem}

A lemma concerning the decay of the coefficients related to the orthonormal basis in 
Lemma \ref{wave} is given below.
\begin{lem}[\cite{Grafakos2015}]\label{smooth}
 Suppose $\si(\xi,\eta)$ defined on $\mathbb R^{2n}$ satisfies that there exists a constant $C_M$ such that 
$\|\p^\al(\si(\xi,\eta))\|_{L^{\nf}}\le C_{M}$ for each multiindex  $|\al|\le M$,
where $M$ is the number of vanishing moments of $\psi_M$. 
Then for 
any 
nonnegative
integer $\ga \in \mathbb N_0=\{n\in\mathbb Z: n\ge 0\}$ we have 
\begin{equation}\label{887}
|\langle \Psi^{\ga,G}_{\vec \mu}, \si\rangle| \leq C C_M2^{-(M+n)\ga} \, .
\end{equation}
\end{lem}


This lemma can be proved by applying Appendix B.2 in \cite{Grafakos2014a}, and we delete the  
details 
which can be found in \cite{Grafakos2015}.

We now go back to the multipliers and consider their wavelet decompositions.
For this purpose, and to apply Lemma \ref{wave} and Lemma \ref{smooth},
we should study kinds of norms of $m_j$.

\begin{lem}\label{om}
There exists a constant $C$ such that
$$
\|m_j\|_{L^2}\le C2^{-j(\la+1/2)},
$$
and for any multiindex $\al$,
$$
\|\p^{\al}m_j\|_{L^{\nf}}\le C2^{-j(\la-|\al|)}.
$$
\end{lem}

\begin{proof}
A change of variables using polar coordinates implies that
\begin{align*}
\|m_j\|_{L^2}=&\Big(\int_{\bbr^{2n}}|m_j(\xi,\eta)|^2d\xi d\eta\Big)^{1/2}\\
\le& C\big(\int_{1-2^{-j}}^{1-2^{-j-2}}(1-r^2)^{2\la}r^{2n-1}dr\big)^{1/2}\\
\le& C(2^{-2j\la}2^{-j})^{1/2}\\
=&C2^{-j(\la+1/2)}
\end{align*}

To estimate the $\al$-th derivatives, we use the Leibniz's rule to write
$$
\p^{\al}m_j(\xi,\eta)=\sum_{\al_1+\al_2=\al} C_{\al_1}\p^{\al_1}m(\xi,\eta)\p^{\al_2}\psi_j(|(\xi,\eta)|) .
$$
Noticing that 
 $|\p^{\al_1}m(\xi,\eta)|\le C 2^{-j(\la-|\al_1|)}$  and
$\p^{\al_2}\psi_j(|(\xi,\eta)|)\le C2^{j|\al_2|}$,   we  derive the bound
$\p^{\al}m_j(\xi,\eta)$ by $C2^{-j(\la-|\al|)}$.
\end{proof}

Unfortunately we need a large number of derivatives in our analysis, and 
the estimate we have on $\|\p^{\al}m_j\|$ is not suitable. 
However, we observe that the support of $m_j$ is very thin, which
is an advantage we should make use of. To realize this, we may dilate it to get a fixed width
so that we are at a good position to apply the wavelet decomposition implying
a good estimate, since, as we will see later, a uniform bound for derivatives plays an important
role in our theory.

Let us define $M_j(\xi,\eta)=m_j(2^{-j}\xi,2^{-j}\eta)$, which
is supported in the annulus $\{(\xi,\eta)\in\mathbb R^{2n}: {2^{j}-1} \le |(\xi,\eta)| \le   {2^{j}-1/4} \}$, whose width is $3/4$. Based on Lemma \ref{om},
we have the following corollary.
\begin{cor}\label{mm}
$M_j$ defined as above satisfies that
$$
\|M_j\|_{L^2}\le C2^{jn}2^{-j(\la+1/2)},
$$
$$
\|\p^{\al}M_j\|_{L^{\nf}}\le C2^{-j\la} \qq\text{for all multiindex }\al, 
$$
and
$$\nabla m_j(\xi,\eta)=2^j(\nabla M_j)(2^j\xi,2^j\eta).$$

\end{cor}

\begin{proof}
A simple change of variables implies that
\begin{align*}
\|M_j\|_{L^2}=&(\int_{\bbr^{2n}}|m_j(2^{-j}\xi,2^{-j}\eta)|^2d\xi d\eta)^{1/2}\\
=& 2^{jn}\|m_j\|_{L^2}\\
\le &C 2^{jn}2^{-j(\la+1/2)}.
\end{align*}

We control $\p^{\al}M_j$ by $|2^{-j|\al|}(\p^{\al}m_j)(2^{-j}\xi,2^{-j}\eta)|\le C
2^{-j|\al|-j(\la-|\al|)}= C2^{-j\la}$.

The verification of the last identity is straightforward.
\end{proof}

Since the new multiplier $M_j$ is still in $L^2$, we have a wavelet decomposition
using Lemma \ref{wave}, i.e.
\begin{equation}\label{emj}
M_j=\sum a_{\om}\om,
\end{equation}
where the summation is over all $\om=\Psi^{\ga,G}_{\vec \mu}$ in the orthonormal basis 
described in Lemma \ref{wave}, the order of cancellations of $\psi_M$ is $M=4n+6$, and
$a_\om=<M_j,\om>$. Concerning the size of $a_\om$, we have the following estimate,
which is a direct implication of Lemma \ref{smooth}
and Corollary \ref{mm}.

\begin{cor}
The coefficient $a_\om$ related to $\om$ with dilation $\ga$ is bounded by 
$C2^{-j\la}2^{-(M+n)\ga}$.

\end{cor}

Before coming to the proof of Proposition \ref{De}, we make a remark. The functions 
$\psi_F$ and $\psi_M$ have compact supports, and all elements in
a fixed level, which means that they have the same dilation factor $\ga$,
in the basis come from translations of finitely many products, so we can  
classify the elements in the basis into finitely many classes so that
all elements in the same level in each class have disjoint supports.
From now on, we can always assume that the supports of $\om$'s 
related to a given dilation
factor $\ga$ are disjoint.

\section{The Proof of Proposition \ref{De}}

With the wavelet decompositions in hand, we are able to prove Proposition \ref{De}. The proof
is inspired by the square function technique (see \cite{Stein1971a} and \cite{Carbery1983} ) and \cite{Grafakos2015}. We control $T_j$ by two integrals with the diagonal part and the off-diagonal parts. For the diagonal part we have just one term, which can be
handled using product wavelets. For the off-diagonal parts
 we introduce two square operators with each one 
 bounded by the Hardy-Littlewood maximal function and a bounded linear operator. 

We need to decompose $M_j$ further. Take $N$ to be a fixed large enough number 
so that $N/10$ is greater than $d$, the diameters of all $\om$ with dilation factor $\ga=0$.
We write $\om(\xi,\eta)=\om_{\vec \mu}(\xi,\eta)=\om_{1,k}(\xi)\om_{2,l}(\eta)$,
where $\vec\mu=(k,l)$ with $k,l\in \mathbb Z^n$, and denote the corresponding 
coefficient $<\om_{k,l},M_j>$ by $a_{k,l}$. We 
define
\begin{equation}
M_j^1=\sum_{\ga}\sum_{|k|\ge N} \sum_{|l|\ge N} a_{k,l}\om_{1,k}\om_{2,l}
\end{equation}
\begin{equation}
M_j^2=\sum_{\ga}\sum_{k} \sum_{|l|\le N} a_{k,l}\om_{1,k}\om_{2,l}
\end{equation}
\begin{equation}
M_j^3=\sum_{\ga}\sum_{|k|\le N} \sum_{|l|\ge N} a_{k,l}\om_{1,k}\om_{2,l}.
\end{equation}
Here $M_j^1$ is the diagonal part whose support is away from both $\xi$ and $\eta$ axes,
$M_j^2$ is the off-diagonal part with the support near the $\eta$ axis, and the support of  
$M_j^3$
 is near the $\xi$ axis. Corresponding to $M_j^i$, we  define
 $m^i_j(\xi,\eta)=M_j^i(2^j\xi, 2^j\eta)$ for $i=1,2,3$.
 Denote   
$$
A_{j,t}(f,g)(x)=
\int_{\mathbb R^{2n}}m_j(t\xi,t\eta)\wh f(\xi)\wh g(\eta)e^{2\pi i x\cdot (\xi+\eta)}d\xi d\eta ,
$$
which equals  $\sum_{i=1}^3A_{j,t}^i(f,g)(x)$
with $A_{j,t}^i(f,g)(x)$ associated with $m_j^i$.

We   need the following lemma to handle $A_{j,t}^1(f,g)(x)$.
\begin{lem}\label{gra}
The gradients $\nabla M_j^i$
and $\nabla m_j^i$ are defined pointwisely for $i=1,2,3$. Moreover
we have the expression
\begin{equation}
\nabla M_j^i(\xi,\eta)=\sum_{\ga}\sum_{k,l}
a_{k,l}\nabla_{(\xi,\eta)}(\om_{1,k}\otimes\om_{2,l})(\xi,\eta),
\end{equation}
with the explanation that the second summation is over allowed pairs $(k,l)$
related to $M_j^i$.
\end{lem}

\begin{proof}
It suffices to verify that  $\p_{\xi_1}M_j^1(\xi,\eta)$ exists for any $(\xi,\eta)$.
We claim that 
$$
\lim_{h\to 0}\f{M_j^1(\xi+he_1,\eta)-M_j^1(\xi,\eta)}{h}
=\sum_{\ga}\sum_k\sum_{|l|\ge N}a_{k,l} 
\p_{\xi_1}\om_{1,k}(\xi)\om_{2,l}(\eta),
$$
where $e_1=(1,0,\dots, 0)\in \rn$. We call on the Lebesgue dominated convergence theorem,
so what we need to show actually is that for $h$ small enough, there exists a constant
$C$ independent of $h$ such that
$|\f{M_j^1(\xi+he_1,\eta)-M_j^1(\xi,\eta)}{h}|\le C$.

We split the levels depending on
whether $h$ is much smaller than $C2^{-\ga}$, the diameters of the supports of $\om$'s.
If $|h|\le C2^{-\ga}$, then the difference 
$\om_{1,k}(\xi+he_1,\eta)-\om_{1,k'}(\xi,\eta)$ with $|k-k'|\le 1$ is 
controlled by $C|h|(|\p_1\om_{1,k}(\xi,\eta)|+|\p_1\om_{1,k'}(\xi,\eta)|)$, which is 
dominated by $C|h|2^{\ga}2^{\ga n/2}$. Recall  the disjointness of the supports of $\om$'s, 
so in each level, there exists at most one $\om$ such that $\om(\xi,\eta)\neq0$.
Consequently,
\begin{align*}
&|M_j^1(\xi+he_1,\eta)-M_j^1(\xi,\eta)|\\
=&\Big|\sum_{C2^{-\ga}\ge |h|}a_{k,l}[\om_{1,k}(\xi+he_1)-
\om_{1,k'}(\xi)]\om_{2,l}(\eta)\\
&+\sum_{C2^{-\ga}\le |h|}[a_{k_1,l_1}\om_{1,k_1}(\xi+he_1)\om_{2,l_1}(\eta)-
a_{k_2,l_2}
\om_{1,k_2}(\xi)\om_{2,l_2}(\eta)]\Big|\\
\le& \sum_{C2^{-\ga}\ge |h|}C2^{-j\la} 2^{-(M+n)\ga}2^{\ga n}2^{\ga}|h|+
\sum_{C2^{-\ga}\le |h|} C2^{-j\la} 2^{-(M+n)\ga}2^{\ga n}\\
\le &C2^{-j\la}(|h|+|h|^M)\\
\le& C2^{-j\la}|h|.
\end{align*}
This concludes the proof of Lemma \ref{gra}.
\end{proof}

 For $f,g\in \mathcal S(\rn)$, using 
 Corollary \ref{mm}, we can rewrite
 \begin{align*}
 &A^1_{j,t}(f,g)(x)\\
 =&\int\int m_j^1(t\xi,t\eta)\wh f(\xi)\wh g(\eta)
 e^{2\pi ix\cdot (\xi+\eta)}d\xi d\eta\\
 =&\int_0^t\int\int (s\xi,s\eta)\cdot\nabla m_j^1(2^js\xi,2^js\eta)\wh f(\xi)\wh g(\eta)
 e^{2\pi ix\cdot (\xi+\eta)}d\xi d\eta\f{ds}s\\
 =&\int_0^t\int\int (2^js\xi,2^js\eta)\cdot\nabla M_j^1(2^js\xi,2^js\eta)\wh f(\xi)\wh g(\eta)
 e^{2\pi ix\cdot (\xi+\eta)}d\xi d\eta\f{ds}s,
 \end{align*}
 where the existence of $\nabla M_j^1$
 and $\nabla m_j^1$ are ensured by Lemma \ref{gra}.
 
 Define the operator $\tilde B_{j,s}^1(f,g)(x)$ related to $( s\xi, s\eta)\cdot\nabla M_j^1( s\xi, s\eta)$
 as
 $\int\int ( s\xi, s\eta)\cdot\nabla M_j^1( s\xi, s\eta)\wh f(\xi)\wh g(\eta)
 e^{2\pi ix\cdot (\xi+\eta)}d\xi d\eta$. Then we have the pointwise estimate
 \begin{equation}\label{Tj1}
 T^1_j(f,g)(x)=\sup_{t>0} |A^1_{j,t}(f,g)(x)|\le \int_0^{\nf}|\tilde B_{j,s}^1(f,g)(x)|\f{ds}s
 \end{equation}
 
We now turn to the study of the boundedness of  $\tilde B_{j,s}^1$.
We set 
$\tilde B_{j,s,\ga}^1(f,g)(x)=\int\int ( s\xi, s\eta)\cdot\nabla M_{j,\ga}^1( s\xi, s\eta)\wh f(\xi)\wh g(\eta)
 e^{2\pi ix\cdot (\xi+\eta)}d\xi d\eta,$
 where
$\nabla M_{j,\ga}^1(\xi,\eta)=\sum_k\sum_la_{k,l} 
\nabla_{(\xi,\eta)}(\om_{1,k}\otimes\om_{2,l})(\xi,\eta).
$
 For the standard case $s=1$, we have the following esitmate.

\begin{prop}\label{BL}
For  $\tilde B^1_{j,1,\ga}$ we have the estimate
$$
\|\tilde B^1_{j,1,\ga}(f,g)\|_{L^1}\le
C C(j,\ga)\|\wh f\chi_{E}\|_{L^2}
\|\wh g\chi_{E}\|_{L^2},
$$
with $C(j,\ga)=C2^{-j(5\la-4n-3)/5}2^{-\ga
(M-4n-5)/5}$, where the set
$E$ is defined as $\{\xi:C2^{-\ga}\le|\xi|\le 2^j\}$.

\end{prop}

In the estimate of the diagonal part in \cite{Grafakos2015}
we use mainly (i) the size of $a_{k,l}$, (ii) the disjointness of the supports of $\om$'s, and (iii)
that $\om$ is a tensor product. So it is easy to obtain Proposition \ref{BL} by examining the proof
in \cite{Grafakos2015} carefully. We sketch the proof here for the sake of completeness.
\begin{proof}
It suffices to consider, for example, $\xi_1\p_{\xi_1}M_{j,\ga}^1(\xi,\eta)$, and by Lemma \ref{gra}, 
this is 
$\sum_k\sum_{l}a_{k,l} 
\p_{\xi_1}\om_{1,k}(\xi)\xi_1\om_{2,l}(\eta)$ for allowed $k,\ l$ in $M^1_{j}$. We rewrite
this as $\sum_k\sum_{l}a_{k,l} 
v_k(\xi)\om_{2,l}(\eta)$, where $v_k(\xi)=\p_{\xi_1}\om_{1,k}(\xi)\xi_1$,
which has the same support as that of $\om_{1,k}$. We denote by $\|a\|_{\nf}$ the $\ell^{\nf}$
norm of $\{a_{k,l}\}_{k,l}$ when $\om_{k,l}$ has the dilation factor $\ga$. 

Define for $r\ge 0$ the set 
$$
U_r=\{(k,l)\in\mathbb Z^{2n}: 2^{-r-1}\|a\|_{\nf}<\|a_{k,l}|\le 2^{-r}\|a\|_{\nf}\},
$$
which has the cardinality at most $C\|a\|_{2}^2\|a\|_{\nf}^22^{-2r}$, where $\|a\|_2$ is the 
$\ell^2$ norm of $\{a_{k,l}\}$ bounded by $\|M_j\|_{L^2}$.
Let $N_1=(2^r\|a\|_{2}/\|a\|_\nf)^{2/5}$,
and define
$$
U_r^1=\{(k,l)\in U_r: \text{card} \{s:(k,s)\in U_r\}\ge N_1\},
$$
$$
U_r^2=\{(k,l)\in U_r\backslash U_r^1: \text{card} \{s:(s,l)\in U_r\backslash U_r^1\}\ge N_1\},
$$
and $U^3_r$ as the remaining set in $U_r$.

For $U^1_r$ we define a related set $E=\{k: \exists\ l\ \text{s.t.} (k,l)\in U^1_r\}$, and its cardinality
$N_2=|E|$ is bounded by $C(N_1)^4$.
We denote by $S_1$ the bilinear operator related to the multiplier 
$\sum_{(k,l)\in U^1_r}a_{k,l}v_k\om_{2,l}$. By an argument similar to \cite{Grafakos2015}
using the three facts we mentioned 
before this proof  we see that
$\|S_1(f,g)\|_{L^1}\le CN_1^22^{\ga(n+1)}2^j2^{-r}\|a\|_{\nf}\|f\|_{L^2}\|g\|_{L^2}$.
We can similarly define $S_2$ and get the same estimate. 
We can classify  $U_r^3$ into $N_1^2$ classes so that for $(k,l), \ (k',l')$ in the same
class with $(k,l)\neq (k',l')$, we must have $k\neq k'$ and $l\neq l'$. This observation enables
us to bound the norm of $S_3$, the operator related to $U_r^3$, by
$CN_1^22^{\ga(n+1)}2^j2^{-r}\|a\|_{\nf}$. 

Using the expression of $N_1$, we obtain that 
$$
\|(S_1+S_2+S_3)(f,g)\|_{L^1}\le C\|a\|_2^{4/5} \|a\|^{1/5}_{\nf}2^{\ga(n+1)}2^j2^{-r/5}
\|f\|_{L^2}\|g\|_{L^2}.
$$
Summing over $r$ and the observation that
in the support of $M_{j,\ga}^1$ we have $\xi\in E_1=\{\xi:2^{-\ga}\le |\xi|\le 2^j\}$
 imply that
$$
\|\tilde B^1_{j,1,\ga}(f,g)\|_{L^1}\le
C \|M_j\|_{L^2}^{4/5}\|a\|_{\nf}^{1/5}2^{\ga(n+1)}2^j\|\wh f\chi_{E}\|_{L^2}
\|\wh g\chi_{E}\|_{L^2},
$$
where the bound is controlled by
$$C(j,\ga)= C2^{-j(5\la-4n-3)/5}2^{-\ga
(M-4n-5)/5}
$$
since $\|M_j\|_{L^2}\le C2^{jn}2^{-j(\la+1/2)}$
by Corollary \ref{mm}.

\end{proof}

Notice that we will have enough decay in $j$ and $\ga$ if
$\la>(4n+3)/5$ and $M>4n+5$, which are satisfied by our assumptions on $\la$ and
choices of wavelets.
\begin{cor}\label{dj}
For the diagonal part we have
$$
\|T_j^1(f,g)\|_{L^1}\le CC(j)\|f\|_{L^2}\|g\|_{L^2},
$$
where $C(j)=\sum_{\ga}C(j,\ga)(j+\ga)$.
\end{cor}

\begin{proof}
From \eqref{Tj1} we know that
$$\|T_j^1(f,g)\|_{L^1}\le
\int_0^{\nf}\|\tilde B^1_{j,s}(f,g)\|_{L^1}\f{ds}{s}
\le 
\int_0^{\nf}\sum_{\ga}\|\tilde B^1_{j,s,\ga}(f_s,g_s)\|_{L^1}\f{ds}{s},
$$
where $\wh f_s(\xi)=s^{-n/2}\wh f(\xi/s)$. Applying Lemma \ref{BL}, the last integral is dominated by
\begin{align*}
& C 
\int_0^{\nf}\sum_{\ga}C(j,\ga) \|\wh f_s\chi_{E_{1}}\|_{L^2}\|\wh g_s\chi_{E_{1}}\|_{L^2}\f{ds}s\\
\le &C\sum_{\ga} C(j,\ga)\big(\int_{\rn}\int_0^\nf|\wh f_s\chi_{E_1}|^2\f{ds}sd\xi\big)^{1/2}
\big(\int_{\rn}\int_0^\nf|\wh g_s\chi_{E_1}|^2\f{ds}sd\xi\big)^{1/2}.
\end{align*}
The double integral involving $f_s$ is bounded by
$\int|\wh f(\xi)|^2\int_{C2^{-\ga}/|\xi|}^{2^j/|\xi|}\tf{ds}sd\xi$, which is less than
$C(j+\ga)\|f\|_{L^2}^2$.
Hence the last sum over $\ga$ is controlled by
$C\sum_{\ga}C(j,\ga)(j+\ga)\|f\|_{L^2}\|g\|_{L^2}$.
\end{proof}

This concludes the argument of the diagonal part. We next deal with the 
the off-diagonal parts. More specifically, we  consider 
$A_{j,t}^2$ since the analysis of $A_{j,t}^3$ is similar.
Recall that 
$$A_{j,t}^2=\int_{\mathbb R^{2n}}m^2_j(t\xi,t\eta)\wh f(\xi)\wh g(\eta)e^{2\pi i x\cdot (\xi+\eta)}d\xi d\eta.$$
We denote $(\xi,\eta)\cdot(\nabla m_j^2)(\xi,\eta)$ by 
$\widetilde m_j^2(\xi,\eta)$.
Then similar to $A^2_{j,t}(f,g)(x)$ we define 
  $\tilde A^2_{j,t}(f,g)(x)$ as 
 $\int\int \wtd m_j^2(t\xi,t\eta)\wh f(\xi)\wh g(\eta)
 e^{2\pi ix\cdot (\xi+\eta)}d\xi d\eta$.
 Like before we can define $B^2_{j,s}$ and $\tilde B^2_{j,s}$ similar to 
 $A^2_{j,s}$ and $\tilde A^2_{j,s}$ with the appearances of $m$ replace by $M$.
 A simple calculation shows that $A^2_{j,s}=B^2_{j,2^js}$, and $\tilde A^2_{j,s}=\tilde B^2_{j,2^js}$.
With all these notations, we have
\begin{align*}
(A^2_{j,t}(f,g)(x))^2=&2\int_0^tA^2_{j,s}(f,g)(x)s\f{dA^2_{j,s}(f,g)(x)}{ds}\f{ds}{s}\\
=&2\int_0^tB^2_{j,2^js}(f,g)(x)\tilde B^2_{j,2^js}(f,g)(x)\f{ds}{s}\\
\le &2\int_0^{\nf}|B^2_{j,s}(f,g)(x)||\tilde B^2_{j,s}(f,g)(x)|\f{ds}{s}\\
\le &2 G_{j,s}(f,g)(x)\wtd G_{j,s}(f,g)(x),
\end{align*}
where we set 
\begin{align*}
G_{j}(f,g)(x)& =\bigg(\int_0^{\nf}|B^2_{j,s}(f,g)(x)|^2\f{ds}{s}\bigg)^{1/2} \\
\wtd G_{j}(f,g)(x)& =     \bigg(\int_0^{\nf}|\tilde B^2_{j,s}(f,g)(x)|^2\f{ds}{s}\bigg)^{1/2}.
\end{align*}  
These 
$g$-functions are bounded from $L^2\times L^2$ to $L^1$.

\begin{lem}\label{gfn}
There exists a constant $C$ independent of $j$ such that for all $f,g\in \mathcal S(\rn)$,
$$
\|G_j(f,g)\|_{L^1}\le C2^{-j(\la+1/2)}\|f\|_{L^2}\|g\|_{L^2}
$$
and
$$
\|\tilde G_j(f,g)\|_{L^1}\le C2^{-j(\la-1/2)}\|f\|_{L^2}\|g\|_{L^2}.
$$
\end{lem}

As in Lemma \ref{BL},  we will
delete some details which can be found in 
\cite{Grafakos2015}.
\begin{proof}
We will focus on $\tilde G_{j}$ first.
For $\tilde G_{j}$ we need to consider two typical cases, the derivative falling on $\xi$
and the derivative falling on $\eta$.

Let us consider the operator with the multiplier $\xi_1\p_{\xi_1}M_j^2$, which equals
$\sum_{\ga}\sum_{k,l}
a_{k,l}v_{k}(\xi)\om_{2,l}(\eta)$ with 
$v_{k}(\xi)=\xi_1\p_{\xi_1}\om_{1,k}(\xi,\eta)$. The corresponding
$g$-function will be denoted by $\tilde G_j^1(f,g)$.
Let us fix $\ga$ and for each $\ga$ at most $N$ number of $\om_{2,l}$ are involved, so we can consider for a single fixed $l$.
Observe that
\begin{align*}
&\int_{\bbr^{2n}}\sum_ka_{k,l}v_k(\xi)\om_{2,l}(\eta)\wh f(\xi)\wh g(\eta)e^{2\pi ix\cdot (\xi,\eta)}
d\xi d\eta\\
=&\|a\|_{\nf}2^{\ga(n+2)/2}2^j\int_{\rn}\om_{2,l}\wh g(\eta)e^{2\pi ix\cdot \eta}d\eta
\int_{\rn}\f{\sum_ka_{k,l}v_k(\xi)}{\|a\|_{\nf}2^{\ga(n+2)/2}2^j}\wh f(\xi)e^{2\pi ix\cdot\xi}d\xi.
\end{align*}
By $|v_k|\le 2^{\ga(n+2)/2}2^j$, 
$|a_{k,l}|\le \|a\|_{\nf}$, and  the disjointness of the supports of $v_k$,
we know that
$(\sum_ka_{k,l}v_k(\xi))/(\|a\|_{\nf}2^{\ga(n+2)/2}2^j)$ is a 
compactly supported bounded function. Hence we bound
the operator related to the multiplier $\sum_ka_{k,l}v_k(\xi)\om_{2,l}(\eta)$
by $C\|a\|_{\nf}2^{\ga(n+1)}2^jM(g)(x)T_{\si}(x)$, where
$T_{\si}(f)$ satisfies that
$\|T_{\si}(f)\|_{L^2}\le C\|\wh f\chi_{F}\|_{L^2}$ with 
$F=\{\xi\in \mathbb R^n: 2^j-1\le |\xi|\le 2^j-1/4\}$.

The operator $\tilde G^1_j$ is then bounded. Indeed we can estimate it by 
a standard dilation argument as follows.
 \begin{align*}
 &\int_{\mathbb R^n}\tilde G^1_{j}(f,g)(x)dx\\
 =&\int_{\mathbb R^n}\bigg(\int_0^{\nf}|
 \int_{\bbr^{2n}}\sum_{\ga}\sum_{k,l}
a_{k,l}v_{k}(s\xi)\om_{2,l}(s\eta)\wh f(\xi)\wh g(\eta)e^{2\pi ix\cdot (\xi,\eta)}
d\xi d\eta|^2\f{ds}s\bigg)^{\frac12}dx\\
\le &C \sum_{\ga}\|a\|_{\nf}2^j2^{(n+1)\ga}\|M(g)\|_{L^2}\int_{\mathbb R^n}\bigg(\int_0^{\nf}s^{-n}
|\wh f(\xi/s) |^2\chi_{F}(\xi) d\xi\frac{ds}{s} \bigg)^{\frac12}\\
\le &  C\sum_{\ga}\|a\|_{\nf}2^j2^{(n+1)\ga}\|g\|_{L^2}\Big(\int_{\mathbb R^n}|\wh f(\xi)|^2
\int_{(2^j-1)/|\xi|}^{(2^j-1/4)/|\xi|}\f{ds}sd\xi\Big)^{\frac12}
\end{align*}
The integral with respect to $s$ is $\log\f{2^j-1/4}{2^j-1}\le C2^{-j}$. This, combined
with the bound of $\|a\|_{\nf}$, shows that the last 
summation is smaller than
\begin{equation}\label{gj1}
C\sum_{\ga}\|a\|_{\nf}2^j2^{(n+1)\ga}2^{-j/2}\|g\|_{L^2}\|f\|_{L^2}\le C2^{-j(\la-1/2)}\|g\|_{L^2}\|f\|_{L^2}.
\end{equation}

For the case the derivative falls on $\eta$, for example $\eta_1$, 
we have a similar representation
\begin{align*}
&\int_{\bbr^{2n}}\sum_ka_{k,l}\om_{1,k}(\xi)v_l(\eta)\wh f(\xi)\wh g(\eta)e^{2\pi ix\cdot (\xi,\eta)}
d\xi d\eta\\
=&\|a\|_{\nf}2^{\ga n/2}\int_{\rn}v_l(\eta)\wh g(\eta)e^{2\pi ix\cdot \eta}d\eta
\int_{\rn}\f{\sum_ka_{k,l}\om_{1,k}(\xi)}{\|a\|_{\nf}2^{\ga n/2}}\wh f(\xi)e^{2\pi ix\cdot\xi}d\xi.
\end{align*}
The first integral in the last line is  dominated by
$2^{\ga n/2}M(g)(x)$ because for the $\om_{2,l}$ with $\ga=0$, we have both
$\p_1(\om_{2,l})^{\vee}(x)e^{2\pi ix\cdot l}$ and $(\om_{2,l})^{\vee}(x)l_1e^{2\pi ix\cdot l}$
are Schwartz functions, and the number of the second type of functions is finite because 
$|l|\le N$. The operator related to the multiplier
$\sum_ka_{k,l}\om_{1,k}(\xi)v_l(\eta)$
is therefore bounded by the quantity 
$C\|a\|_{\nf}2^{\ga(n+1)}M(g)(x)T_{\si'}(f)(x)$, where $T_{\si'}$ satisfies the same property
$T_\si$ has.
For $L^1$ norm of the $g$-function $\tilde G^2_j$ related the multiplier $\sum_ka_{k,l}\om_{1,k}(\xi)v_l(\eta)$
we apply a similar argument used in estimating $\|\tilde G^1_j\|_{L^1}$ to control it
by
\begin{equation}\label{gj2}
C\sum_{\ga}2^{-j(\la+1/2)}2^{-(M-1)\ga}
\|g\|_{L^2}\|f\|_{L^2}
\le  C2^{-j(\la+1/2)}\|f\|_{L^2}\|g\|_{L^2}
\end{equation}
This esimate and
\eqref{gj1}
show that 
$$
\|\tilde G_j(f,g)\|_{L^1}\le C2^{-j(\la-1/2)}\|f\|_{L^2}\|g\|_{L^2}.
$$

For $G_j(f,g)$, a similar and simpler argument applying to the standard representation 
$\sum a_{k,l}\om_{1,k}(\xi)\om_{2,l}(\eta)$ gives that
$$
\|G_j(f,g)\|_{L^1}\le C2^{-j(\la+1/2)}\|f\|_{L^2}\|g\|_{L^2}.
$$
Here the difference $2^j$ comes from the fact that in the multiplier of $B_{j,s}^2$, we 
miss the term $(\xi,\eta)$, which is just controlled by $2^j$.
\end{proof}

\begin{cor}\label{odj}
For the off-diagonal part we have
$$
\|T_j^2(f,g)\|_{L^1}\le C2^{-j\la}\|f\|_{L^2}\|g\|_{L^2}.
$$
\end{cor}
\begin{proof}
By the calculation before Lemma \ref{gfn} we have the pointwise control
$T^2_j(f,g)(x)\le \sqrt2 (G_{j}(f,g)(x)\wtd G_{j}(f,g)(x))^{1/2}$, which and Lemma 
\ref{gfn} imply that
\begin{align*}
\|T^2_j(f,g)\|_{L^1}\le &\|\sqrt2 ( G_j(f,g)(x)\wtd  G_j(f,g)(x))^{1/2}\|_{L^1}\\
\le &C(\| G_j(f,g)(x)\|_{L^1}\|\wtd  G_j(f,g)(x)\|_{L^1})^{1/2}\\
\le & C(2^{-j(\la+1/2)}2^{-j(\la-1/2)}\|f\|_{L^2}^2\|g\|_{L^2}^2)^{1/2}\\
= & C2^{-j\la}\|f\|_{L^2}\|g\|_{L^2}
\end{align*}
In this case we have nice decay for $T_j^2$ since 
$\la>(4n+3)/5>0$. 
\end{proof}

We now can prove Proposition
\ref{De}, the key result of our theory.
\begin{proof}[Proof of Proposition
\ref{De}]
When $\la>(4n+3)/5$, the simple observation $T_j\le T^1_j+T^2_j+T^3_j$, the Corollary \ref{dj}, and Corollary \ref{odj} 
complete the proof of Proposition \ref{De}.
\end{proof}

\section{A Final Remark}

Tao \cite{Tao1998} proves that a necessary condition so that the linear maximal
Bochner-Riesz operator $B^\la_*$ is bounded on $L^p(\rn)$
is that $\la\ge\tf{2n-1}{2p}-\tf n2$. We modify his argument  in this section to show that
a similar requirement is also needed in  the bilinear
setting.

\begin{prop}\label{06152}
A necessary condition such that the bilinear maximal Bochner-Riesz operator $T^\la_*$  is bounded from $L^{p_1} (\rn)\times L^{p_2}(\rn)$ to weak $L^p(\rn)$
is that $\la\ge\tf{2n-1}{2p}-\tf {2n-1}2$.

\end{prop}

\begin{rmk}
This result is meaningful only if $p<1$.
We observe that the kernel requires  another necessary condition  
 $(\la+(2n+1)/2)p\ge n$, i.e.,  $\la\ge\tf np-\tf{2n+1}2$,
 which is less restrictive than $\la\ge\tf{2n-1}{2p}-\tf {2n-1}2$.
\end{rmk}

\begin{proof}
We prove this theorem by constructing a counterexample. Let $M$
be a large number and $\ep$ a small number. Define a smooth function $\vp(x)=\vp_{\ep,M}(x)=\psi(\ep^{-1}|x'|)
\psi(\ep^{-1}M^{-1/2}x_n)$, where
$x'=(x_1,\dots, x_{n-1})$ and $\psi$ is a smooth bump function
supported in the interval $[-1,1]$. Define $f_M(y)=e^{2\pi i y_n}\vp(y)$
and $S_M=\{x: M\le|x'|\le 2M, M\le x_n\le 2M\}$.
Obviously we have $\|f_M\|_p\sim (\ep^nM^{1/2})^{1/p}$ and $|S_M|\sim M^n$.

We will show that $T^\la_*(f_M,f_M)(x)$ is bounded from below for $x\in S_M$. 
Let us take $R=R_x=\sqrt 2|x|/x_n$, which is comparable to $1$. 
Recall that the kernel $K^\la$ has the asymptotic representation for $X\in\bbr^{2n}$, as 
$|X|\to\nf$,
\begin{align*}
K^\de(X)=&\f{\Ga(\la+1)}{\pi^{\la}}\f{J_{\la+n}(2\pi |X|)}{|X|^{\la+n}}\\
=&Ce^{2\pi i|X|}|X|^{-(\la+\tf{2n+1}2)}+Ce^{-2\pi i|X|}|X|^{-(\la+\tf{2n+1}2)}+O(|X|^{-(\la+\tf{2n+3}2)}).
\end{align*}
From this and $R$ is comparable to $1$ we can control $A^\la_{1/R}(f_M,f_M)(x)$
from below by
\begin{align*}
&C_1\bigg|\int_{\rn}\int_{\rn} e^{2\pi i R|(x,x)|}|(x-y,x-z)|^{-(\la+\tf{2n+1}2)}\vp(y)\vp(z)dydz\bigg|\\
&-C_2\bigg|\int_{\rn}\int_{\rn}(e^{2\pi i (R|(x-y,x-z)|+y_n+z_n)}-e^{2\pi iR|x|})|(x-y,x-z)|^{-(\la+\tf{n+1}2)}\vp(y)\vp(z)dydz\bigg|\\
&-C_3\bigg|\int_{\rn}\int_{\rn} e^{2\pi i (-R|(x-y,x-z)|+y_n+z_n)}|(x-y,x-z)|^{-(\la+\tf{n+1}2)}\vp(y)\vp(z)dydz\bigg|\\
&-C_4\bigg|\int_{\rn}\int_{\rn} |(x-y,x-z)|^{-(\la+\tf{n+3}2)}|\vp(y)\vp(z)|dydz\bigg|\\
&=C_1I-C_2II-C_3III-C_4IV.
\end{align*}
The first term is the main term, for which we have
$I\sim M^{-(\la+\tf{2n+1}2)}\ep^{2n}M$. For the second term, we use that
$|R|(x-y,x-z)|+y_n+z_n-R|(x,x)||\le C\ep$ to control it by
$ M^{-(\la+\tf{2n+1}2)}\ep^{2n+1}M$. 
To handle the third term, we apply integration by parts to the variable $y_n$
and can show that it is bounded by 
$M^{-(\la+\tf{2n+1}2)}\ep^{2n}M(M^{-1}+\ep^{-1}M^{-1/2})$. 
It is not hard to see that
$IV\le M^{-(\la+\tf{2n+3}2)}\ep^{2n+1}M$.
So we see that
$|T^\la_*(f_M,f_M)(x)|\ge C_1M^{-(\la+\tf{2n-1}2)}\ep^{2n}[1-C_2'(\ep+M^{-1}+\ep^{-1}M^{-1/2})]$.
So if we choose $\ep$ to be a fixed small number and let $M\to\nf$, the last quantity 
is comparable to $M^{-(\la+\tf{2n-1}2)}$. Set this number to be $\al$ and if $T^\la_*$
is bounded from $L^{p_1} (\rn)\times L^{p_2}(\rn)$ to weak $L^p(\rn)$, we  have
$M^n\sim|S_M|\le|\{x:|T^\la_*(f_M,f_M)(x)|\}\ge \al|\le \|f_M\|_{p_1}^p\|f_M\|_{p_2}^p/\al^p\le M^{\tf 12+(\la+\tf{2n-1}2)p}$,
which gives the necessary condition $\la\ge\tf{2n-1}{2p}-\tf {2n-1}2$.
\end{proof}

What we have proved in Theorem  \ref{06151} and Theorem  \ref{Main} are far from
the restriction given by Proposition \ref{06152}, so is
$T^\la_*$ actually  bounded from $L^{p_1} (\rn)\times L^{p_2}(\rn)$ to weak $L^p(\rn)$
for $p>\max(\tf12,\tf{2n-1}{2\la+2n-1})$? Or do we  at least have that
$T^\la_*$ is bounded from 
$L^{2} (\rn)\times L^{2}(\rn)$ to  $L^1(\rn)$ when $\la>0$?

\bigskip

\noindent {\bf Acknowledgement:}
The author would like to thank his advisor Professor L. Grafakos for binging this problem
to his attention
and for many   valuable comments.

 

 \end{document}